\documentclass{amsart}
\usepackage{amsrefs}
\usepackage{xcolor}
\usepackage{amssymb}
\usepackage{hyperref}
\usepackage{enumerate}

\DeclareMathOperator{\supp}{supp}
\DeclareMathOperator{\rank}{rank}
\newtheorem{theorem}{Theorem}[section]

\newtheorem{remark}{Remark}
\usepackage[utf8x]{inputenc}

\hypersetup{pdfstartview=FitBH}

\DeclareFontFamily{U}{mathx}{\hyphenchar\font45}
\DeclareFontShape{U}{mathx}{m}{n}{
      <5> <6> <7> <8> <9> <10>
      <10.95> <12> <14.4> <17.28> <20.74> <24.88>
      mathx10
      }{}
\DeclareSymbolFont{mathx}{U}{mathx}{m}{n}
\DeclareFontSubstitution{U}{mathx}{m}{n}
\DeclareMathAccent{\widecheck}{0}{mathx}{"71}
\DeclareMathAccent{\wideparen}{0}{mathx}{"75}

\begin{document} 

\title[Maximal inequalities and decay of Fourier transforms]{Maximal inequalities and the decay of Fourier transforms of measures}

\author{Terence L.~J.~Harris}
\address{School of Mathematics and Physics\\ The University of Queensland\\ St Lucia QLD 4067\\  Australia}
\email{terry.harris@uq.edu.au}

\begin{abstract} It is shown that Schrödinger maximal inequalities over fractals are equivalent to the $L^2$ decay rates of Fourier transforms of fractal measures over the paraboloid. A similar connection is shown between the wave equation and cone averages. 

One implication is well-known and follows from the Kolmogorov-Seliverstov-Plessner method, but the other implication is nontrivial and relies on a variant of the Marstrand projection theorem. 

The idea of the proof is to insert an extra averaging parameter into a proof of Lucà and Rogers, which used a quantitative ergodic lemma instead of the Marstrand projection theorem.  Lucà and Rogers gave a second proof of Bourgain's necessary condition $s\geq \frac{n}{2(n+1)}$ for Schrödinger solutions in $\mathbb{R}^{n+1}$ to converge pointwise a.e.~ back to the initial data as time tends to zero.  One application of the main theorem in this article is a proof of Bourgain's necessary condition which does not use ergodic theory or number theory. 
\end{abstract}
\maketitle

\section{Introduction}


For $\alpha \in (0,n+1]$, let $\beta(\alpha, \mathbb{P}^n)$ be the supremum over all $\beta \in \mathbb{R}$ with the property that for some constant $C_{\alpha,\beta}$, for all $R \geq 1$,
\[ \int_{\mathbb{P}^n} \left\lvert \widehat{\nu}( R \xi) \right\rvert^2 \, d\sigma_{\mathbb{P}_n}(\xi) \leq C_{\alpha, \beta} I_{\alpha}(\nu) R^{-\beta}, \]
for all $\mu \in \mathcal{M}\left(B_{n+1}(0,1)\right)$, where $\sigma_{\mathbb{P}_n}$ is the pushforward of Lebesgue measure on $B_n(0,1)$ under $\xi \mapsto (\xi, |\xi|^2)$, and 
\[ I_{\alpha}(\nu) = \int_{\mathbb{R}^{n+1}} \left\lvert \widehat{\nu}(\xi) \right\rvert^2 |\xi|^{\alpha-(n+1)} \, d\xi, \]
 is the Fourier energy of $\nu$ in $\mathbb{R}^{n+1}$. Here $\mathcal{M}(A)$ denotes the set of finite Borel measures compactly supported in $A$. 

Let $s_n(\alpha) = \frac{1}{2}(n-\beta(\alpha, \mathbb{P}^n))$ for $\alpha \in (0,n+1]$. By duality, $s_n(\alpha)$ is the infimum over all $s \in \mathbb{R}$ such that for some constant $C_{\alpha,s}$,
\begin{equation} \label{locals} \int \left\lvert e^{ i \frac{t}{R} \Delta}f(x) \right\rvert \, d\nu(x,t) \leq C_{\alpha, s}I_{\alpha}(\nu)^{1/2} R^s \|f\|_2, \end{equation}
for all $\nu \in \mathcal{M}\left(B_{n+1}(0,1)\right)$ and Schwartz $f$ with $\supp \widehat{f} \subseteq B_n(0,R)$, for all $R \geq 1$, where 
\[ e^{ i t \Delta}f(x) = (2\pi)^{-n/2} \int_{\mathbb{R}^n} e^{i \left(\langle x, \xi \rangle + t\lvert \xi\rvert^2\right)} \widehat{f}(\xi) \, d\xi, \]
is the solution to the linear Schrödinger equation
\begin{equation} \label{freeschrodinger} \begin{cases} &i \partial_t u - \Delta u = 0 \\
&u(x,0) = f(x), \end{cases} \end{equation}
where $(x,t) \in \mathbb{R}^n \times \mathbb{R}$. 


 For $\alpha \in (0,n+1]$, let $s_n^{\max}(\alpha)$ be the infimum over all $s \in \mathbb{R}$ such that for some constant $C_{\alpha,s}$,
\begin{equation} \label{localsmax} \int \sup_{0 < t < 1} \left\lvert e^{ i t \Delta}f(x) \right\rvert \, d\mu(x,t) \leq C_{\alpha,s}I_{\alpha}(\mu)^{1/2} R^s \|f\|_2, \end{equation}
for all $\mu \in \mathcal{M}\left(B_{n}(0,1)\right)$ and Schwartz $f$ with $\supp \widehat{f} \subseteq B_n(0,R)$, where $I_{\alpha}(\mu)$ is the Fourier energy of $\mu$ in $\mathbb{R}^{n}$. 
It is well-known and straightforward to check that $s_n(\alpha)$, $s_{n}^{\max}(\alpha)$ are continuous in $\alpha$ (\cite[Lemma~3.1]{wolffcircle}). Moreover, when $\supp \widehat{f} \subseteq B_n(0,R) \setminus B_n(0,R/2)$, \eqref{localsmax} is known to be equivalent to the same inequality with the $\sup_{0 < t < 1}$ replaced by the smaller $\sup_{0 < t < R^{-1}}$ \cite{leeimrn,lucarogers}, and consequently for $0 < \alpha \leq n$,  the standard Kolmogorov-Seliverstov-Plessner method gives that $s_n^{\max}(\alpha) \leq s_n(\alpha)$.

For both $\beta(\alpha, \mathbb{P}^n)$ and $s_n(\alpha)$, an equivalent definition results if $I_{\alpha}(\nu)$ is replaced by $c_{\alpha}(\nu) \nu(\mathbb{R}^{n+1})$ or the non-Fourier energy when $0 < \alpha < n+1$, and similarly for $s_n^{\max}(\alpha)$ when $0 < \alpha < n$ \cite[Lemma~1.5]{wolffcircle}, where 
\[ c_{\alpha}(\nu) =\sup_{x \in \mathbb{R}^{n+1}, r>0} r^{-\alpha} \nu(B(x,r)). \]
 Below, $\|f\|_{H^s} = \|f\|_{H^s(\mathbb{R}^n)}$ is defined by 
\[ \|f\|_{H^s} = \left( \int_{\mathbb{R}^n} \left( 1+ \lvert \xi \rvert^2 \right)^{s} \left\lvert \widehat{f}(\xi) \right\rvert^2 \, d\xi\right)^{1/2}. \] Let $\mathcal{S}(\mathbb{R}^n)$ be the set of Schwartz functions on $\mathbb{R}^n$.  The main result of this article is the following. 

\begin{theorem} \label{equivalence} Let $n \geq 1$.  Let $\alpha>0$ and let $s \geq  0$. There is a constant $K$  depending only on $s$ and $n$ such that following holds. 
\begin{enumerate}[(i)] 
\item \label{sufficientcondition}   If 
\begin{multline} \label{maxassumed} \int \sup_{0 < t < 1} \left\lvert e^{i t \Delta}g \right\rvert \, d\mu(x) \leq C_{\alpha,s}I_{\alpha}(\mu)^{1/2} \left\lVert g \right\rVert_{H^s(\mathbb{R}^n)} \\
 \forall \mu \in \mathcal{M}\left(B_n(0,1)\right) \quad \forall g \in \mathcal{S}(\mathbb{R}^n), \end{multline}
then
\begin{multline} \label{goal} \int \left\lvert e^{i t \Delta}f \right\rvert \, d\nu(x,t) \leq K C_{\alpha,s}I_{\alpha}(\nu)^{1/2} \left\lVert f \right\rVert_{H^s(\mathbb{R}^n)}\\
 \forall \nu \in \mathcal{M}\left(B_{n+1}(0,1)\right) \quad \forall f \in \mathcal{S}(\mathbb{R})^n. \end{multline}

\item \label{localversion} (Local version)  If $\lambda\geq 1$ and
\begin{multline} \label{maxassumed2} \int \sup_{0 < t < \lambda^{-1}} \left\lvert e^{i t \Delta}g \right\rvert \, d\mu(x) \leq C_{\alpha,s}I_{\alpha}(\mu)^{1/2} \lambda^s \left\lVert g \right\rVert_{2}\\
\forall \mu \in \mathcal{M}\left(B_n(0,1)\right) \quad \forall g \in \mathcal{S}(\mathbb{R}^n) \text{ with } \supp \widehat{g} \subseteq B_n(0, \lambda), \end{multline}
 then 
\begin{multline} \label{goal2} \int \left\lvert e^{i \frac{t}{\lambda} \Delta}f \right\rvert \, d\nu(x,t) \leq K C_{\alpha,s} I_{\alpha}(\nu)^{1/2} \lambda^s \left\lVert f \right\rVert_2\\
\\
 \forall \nu \in \mathcal{M}\left(B_{n+1}(0,1)\right) \quad \forall f \in \mathcal{S}(\mathbb{R})^n \text{ with }  \supp \widehat{f} \subseteq B_n(0, \lambda).\end{multline} 
\item \label{weakversion} The above equivalences all hold if the $L^1(\mu)$ and $L^1(\nu)$ norms on the left-hand side are replaced by $L^{1, \infty}(\mu)$ and $L^{1, \infty}(\nu)$ norms.
\item \label{nontrivial1} For all $\alpha \in (0,n]$, $s_n(\alpha) = s_{n}^{\max}(\alpha)$. If $n <\alpha \leq n+1$ then $s_n^{\max}(\alpha) \geq s_n(\alpha)$.  
\end{enumerate}
 \end{theorem} 

\begin{remark} Clearly \eqref{goal} implies \eqref{maxassumed} and \eqref{goal2} implies \eqref{maxassumed2} (with different constants) when $ 0 < \alpha < n$. Moreover, if \eqref{maxassumed2} holds then the same holds with $\sup_{0 < t< \lambda^{-1}}$ replaced by $\sup_{0 < t < 1}$ for functions with $\supp \widehat{g} \subseteq B_n(0, \lambda) \setminus B_n(0, \lambda/2)$ \cite[Lemma~2.1]{leerogers}. Hence \eqref{maxassumed2} (for all $\lambda$) implies \eqref{maxassumed} with an $\epsilon$ loss in $s$. Therefore, \eqref{maxassumed}, \eqref{goal}, \eqref{maxassumed2}, and \eqref{goal2} are all equivalent up to $\epsilon$ losses in $s$ when $0 < \alpha < n$. \end{remark}

In \cite{eceiz},  Eceizabarrena and Ponce-Vanegas proved a lower bound on $s_n^{\max}(\alpha)$  which for $\alpha \in [n/2,n]$ matches the lower bound on the a priori larger $s_n(\alpha)$ from \cite[Theorem~1.2]{du} (equivalently upper bound on $\beta(\alpha, \mathbb{P}^n)$). Theorem~\ref{equivalence} shows that lower bounds on $s_n^{\max}(\alpha)$ follow automatically from lower bounds on $s_n(\alpha)$.  The problem of determining $s_n^{\max}(\alpha)$ and $s_n(\alpha)$ were referred to in the introduction of \cite{eceiz} as being ``morally equivalent'', but as discussed in \cite{eceiz}, specific counterexamples for $s_n(\alpha)$ may not be counterexamples for $s_n^{\max}(\alpha)$. In \cite{eceiz} they also gave lower bounds for the Hausdorff dimension of the set of $x$ for which $u(x,t) \not \to u_0(x)$, where $u$ solves \eqref{freeschrodinger} with initial data $u_0 \in H^s$. Their bounds do not follow automatically from lower bounds on $s_n^{\max}$ since there is no  maximal principle for fractal measures. 


Theorem~\ref{equivalence} also implies that upper bounds for $s_n^{\max}(\alpha)$ when $0 < \alpha \leq n$ could imply lower bounds for $\beta(\alpha, \mathbb{P}^n)$, but currently the best known upper bounds for the divergence set proceed via lower bounds $\beta(\cdot , \mathbb{P}^n)$. 

The naïve way to try to obtain \eqref{goal} from \eqref{maxassumed} would be to project the measure $\nu$ under $\pi(x,t)=x$ and apply \eqref{maxassumed} to the projection. The obstruction to this approach is that the Fourier energy of the projected measure $\pi_{\sharp} \nu$ of $\nu$ may be much larger than that of $\nu$. A way around this is to change variables in the Schrödinger operator, translating the Fourier transform of the initial data by some vector $-\theta$. This translation idea was used by Lucà and Rogers \cite{lucarogers2}. Due to the quadratic phase, these translations can be passed over to the measure $\nu$, replacing $\nu$ by the modified $T_{\theta \sharp} \nu$ for a linear map $T_{\theta} : \mathbb{R}^{n+1} \to \mathbb{R}^{n+1}$ given by $T_{\theta}(x,t) = (x-2t\theta,t)$. Up to some harmless scalings of the coordinates, $\pi \circ T_{\theta}$ can be thought of as a projection onto the $n$-plane in $\mathbb{R}^{n+1}$ orthogonal to $(2\theta, 1)$, where $\pi(x,t) =x$. Scaling $(2\theta, 1)$ to get a unit vector, this smoothly parametrises the upper hemisphere of sphere $S^n$ as $\theta$ varies.  Kaufman's version of the Marstrand projection theorem in the plane \cite{kaufman}, and Mattila's generalisation to $n \geq 2$ \cite{mattila1975}, imply that for a typical unit vector $v \in S^n$, the ($\alpha$-dimensional) energy of the projected measure $P_{v^{\perp} }\nu$ is bounded by the energy of $\nu$. This works for most values of $\alpha$, but for the case $\alpha=n$ corresponding to a.e.~convergence, it is necessary to use the Fourier energy instead of the standard energy. This Fourier version of the Marstrand projection theorem reduces to the case $k=1$ of the following identity due to Solmon~\cite[Lemma~2.2]{solmon}, for $1 \leq k < d$:
\begin{equation} \label{solmonidentity} \int_{G(d,k)} \int_{V^{\perp}} |x|^k f(x) \, d\mathcal{H}^{d-k}(x)\,  d\gamma_{d,k}(V) = \frac{ \mathcal{H}^{d-k-1}(S^{d-k-1}) }{\mathcal{H}^d(S^{d-1})} \int_{\mathbb{R}^d} f(x) \, dx, \end{equation}
where $\gamma_{d,k}$ is the Haar probability measure on $G(d,k)$. Using polar coordinates, Kaufman \cite{kaufman} proved the Fourier energy version of the Marstrand projection theorem in the plane, and Kaufman's method generalises by using \eqref{solmonidentity}. This higher dimensional generalisation is well-known (see e.g.~the remark after Theorem~5.9 in \cite{mattilafourier}), but I do not know of the earliest reference.

In the proof of Lucà and Rogers, the averaging parameter $\theta$ varies over a set of one dimension less than in this article, and the quantitative ergodic theorem they use seems to be a kind of ``restricted projection theorem''. Adding an additional averaging parameter allows a simpler projection theorem to be used which applies in a more general setting.

The algebraic simplicity of the function $\left\lvert \xi \right\rvert^2$ in the phase of the Schrödinger operator seems to be crucial to the argument. In Section~\ref{wavesection} the proof is extended to the wave equation. The square root in the phase of the wave operator would seem to be an issue, but since paraboloids are conic sections, after a rotation the phase becomes quadratic when one variable is fixed, and the same approach as in Schrödinger case can be applied. 



\section{Proof for the Schrödinger equation}

It follows directly from the definitions that $\eqref{localversion} \Rightarrow \eqref{nontrivial1}$ in Theorem~\ref{equivalence}. 
so it suffices to prove \eqref{sufficientcondition} and \eqref{localversion} (and the weak $L^1$ case mentioned in \eqref{weakversion}). Only the proof of \eqref{localversion} will be given since the proof of \eqref{sufficientcondition} is virtually identical to the case $\lambda = 1$ of \eqref{localversion}, and the proof of the weak $L^1$ version requires only minor adjustments to the beginning of the proof of the strong $L^1$ version.  

\begin{proof}[Proof of \eqref{localversion} of Theorem~\ref{equivalence}] Let $\alpha>0$, $s \geq 0$, and assume \eqref{maxassumed2} holds. It will be shown that \eqref{goal2} holds. 

Let $\nu$ be a Borel probability measure on $B_{n+1}(0,1)$, let $\lambda \geq 1$ and let $f$ be a Schwartz function with $\supp \widehat{f} \subseteq B_n(0, \lambda)$. Without loss of generality it may be assumed that $\nu$ is supported in $0 \leq t \leq 1$. For each $\theta \in B_n(0, \lambda)$, let 
\[ \widehat{f_{\theta}}(\eta) =  \widehat{f}(\eta - \theta). \]
By a change of variables 
\[ e^{it \Delta} f(x) = (2\pi)^{-n/2} \int e^{i \left\langle \left(\eta - \theta, \left\lvert \eta-\theta \right\rvert^2 \right), (x, t) \right\rangle } \widehat{f_{\theta}}(\eta) \, d\eta. \]
This becomes
\begin{multline*} e^{it \Delta} f(x) = (2\pi)^{-n/2} e^{i \left(- \langle \theta, x \rangle + \lvert \theta \rvert^2 t\right)} \int e^{i \left\langle \left( \eta, \left\lvert \eta \right\rvert^2 \right), (x-2t\theta , t) \right\rangle } \widehat{f_{\theta}}(\eta) \, d\eta \\
 =   e^{i \left( -\langle \theta, x \rangle + \lvert \theta \rvert^2 t\right)}e^{it \Delta}f_{\theta}( x-2t\theta ). \end{multline*} 
Hence, letting $D_{\lambda^{-1}}(x,t) = (x, t/\lambda)$,  for any $\theta \in B_n(0,\lambda)$,
\begin{multline*} \int \left\lvert e^{i\frac{t}{\lambda} \Delta}f(x) \right\rvert \, d\nu(x,t)  = \int \left\lvert e^{it \Delta}f(x) \right\rvert \, dD_{\lambda^{-1}\sharp}\nu(x,t) \\
= \int \left\lvert e^{it \Delta}f_{\theta}\left(x-2t\theta\right) \right\rvert \, dD_{\lambda^{-1} \sharp}\nu(x,t). \end{multline*}
If $T_{\theta} : \mathbb{R}^{n+1} \to \mathbb{R}^{n+1}$ is the linear map $T_{\theta}(x,t) = (x-2t\theta , t)$, the above can be written as 
\[ \int \left\lvert e^{i\frac{t}{\lambda} \Delta}f(x) \right\rvert \, d\nu(x,t)  = \int \left\lvert e^{it \Delta}f_{\theta}(x) \right\rvert \, dT_{\theta \sharp} D_{\lambda^{-1} \sharp} \nu(x,t). \]
Hence, for any $\theta \in B_n(0, \lambda)$, 
\begin{multline*} \int \left\lvert e^{i\frac{t}{\lambda} \Delta}f(x) \right\rvert \, d\nu(x,t) \leq \int \sup_{0 < t <\lambda^{-1}} \left\lvert e^{it\Delta}f_{\theta}(x) \right\rvert \, dT_{\theta \sharp} D_{\lambda^{-1} \sharp}\nu(x,t)  \\
= \int \sup_{0 < t <\lambda^{-1}} \left\lvert e^{it \Delta}f_{\theta}(x) \right\rvert \, d\pi_{\sharp} T_{\theta \sharp} D_{\lambda^{-1} \sharp} \nu(x) = \int \sup_{0 < t <\lambda^{-1}} \left\lvert e^{it \Delta}f_{\theta}(x) \right\rvert \, d\Pi_{\theta/\lambda \sharp} \nu(x), \end{multline*}
where $\pi: \mathbb{R}^{n+1} \to \mathbb{R}^n$ is $(x,t) \mapsto x$, and $\Pi_{\theta}(x,t) = x-2t\theta$. By the assumed inequality \eqref{maxassumed2}, using that $\supp \Pi_{\theta/\lambda \sharp} \nu \subseteq B_n(0,100)$, this gives 
\begin{equation} \label{detour} \int \left\lvert e^{i\frac{t}{\lambda} \Delta}f(x) \right\rvert \, d\nu(x,t) \leq C_{\alpha,s}I_{\alpha}\left(\Pi_{\theta/\lambda \sharp} \nu \right)^{1/2} \lambda^s \|f\|_{2}, \end{equation}
where the above used that $\supp \widehat{f_{\theta}} \subseteq B_n(0,2\lambda)$. A formula for the linear map $\Pi_{\theta} : \mathbb{R}^{n+1} \to \mathbb{R}^n$ is 
\[ \Pi_{\theta}(x,t) = x - 2t \theta = \begin{pmatrix} I_n  & -2\theta \end{pmatrix} \begin{pmatrix} x \\ t \end{pmatrix}, \]
where $I_n$ is the $n \times n$ identity matrix,  so $\Pi_{\theta}$ can be identified with the $n \times (n+1)$ matrix $\begin{pmatrix} I_n  & -2\theta \end{pmatrix}$ which has rank $n$. The range of $\Pi_{\theta}^T$ is the orthogonal complement in $\mathbb{R}^{n+1}$ of the kernel of $\Pi_{\theta}$, and the kernel of $\Pi_{\theta}$ is spanned by the unit vector 
\[ G(\theta) = \frac{1}{\sqrt{ 4 \lvert \theta \rvert^2 + 1 } }\begin{pmatrix} 2 \theta \\ 1 \end{pmatrix}, \]
where $G: \mathbb{R}^n \to S^n_+$ is a $C^{\infty}$ diffeomorphism from $\mathbb{R}^n$ to the upper hemisphere $S^n_+$ of $S^n$ with inverse $F: S^n_+ \to \mathbb{R}^n$ given by $F(x,x_{n+1}) = \frac{1}{2x_{n+1}} x$. 

By taking the QR decomposition of $\Pi_{\theta}^T$, the matrix $\Pi_{\theta}$ can be written as $\Pi_{\theta} = L_{\theta} Q_{\theta}$, where $L_{\theta}$ is an $n \times n$ lower triangular matrix with positive entries on the diagonal, and $Q_{\theta}$ is an $n \times (n+1)$ matrix whose rows form an orthonormal basis for the complement of $G(\theta)$ (the range of $\Pi_{\theta}^T$). By uniqueness of this decomposition (using \cite[Theorem~5.2.2]{golubvanloan} and $\rank \Pi_{\theta}^T = n$), $L_{\theta}$ and $Q_{\theta}$ depend continuously on $\theta$. Substituting $\Pi_{\theta} = L_{\theta} Q_{\theta}$ into \eqref{detour} gives 
\[   \int \left\lvert e^{i\frac{t}{\lambda} \Delta}f(x) \right\rvert \, d\nu(x,t) \leq C_{\alpha,s} I_{\alpha}\left(L_{\theta/\lambda \sharp} Q_{\theta/\lambda \sharp} \nu \right)^{1/2} \lambda^s \|f\|_2. \]
Integrating both sides over $B_n(0,\lambda)$, changing variables and applying Cauchy-Schwarz gives
\[ \int \left\lvert e^{i\frac{t}{\lambda} \Delta}f(x) \right\rvert \, d\nu(x,t) \lesssim  C_{\alpha,s}\lambda^s \|f\|_2\left(\int_{B_n(0,1)} I_{\alpha}\left(L_{\theta \sharp} Q_{\theta \sharp} \nu \right) \, d\mathcal{H}^n(\theta)\right)^{1/2}, \]
where $\mathcal{H}^n$ is the Lebesgue measure on $\mathbb{R}^n$ (which equals $n$-dimensional Hausdorff measure). Since $L_{\theta}$ is an invertible $n \times n$ matrix which is continuous in $\theta$, by changing variables in the definition of Fourier energy this can be simplified to 
\begin{equation} \label{aftersimplify} \int \left\lvert e^{i\frac{t}{\lambda}\Delta}f(x) \right\rvert \, d\nu(x,t) \lesssim  C_{\alpha,s}\lambda^s\|f\|_{2}\left(\int_{B_n(0,1)}I_{\alpha}\left(Q_{\theta \sharp} \nu \right) \, d\mathcal{H}^n(\theta)\right)^{1/2}. \end{equation}
But since the rows of $Q_{\theta}$ form an orthonormal basis for the complement of $G(\theta)$, and by rotation invariance of the measure $\mathcal{H}^n$ on $\mathbb{R}^{n+1}$,
\[ I_{\alpha}\left(Q_{\theta \sharp} \nu \right) = \int_{G(\theta)^{\perp}} \left\lvert \xi \right\rvert^{\alpha-n} \left\lvert \widehat{\nu}(\xi) \right\rvert^2 \, d\mathcal{H}^n(\xi). \]
  Therefore, \eqref{aftersimplify} becomes 
\begin{multline*} \int \left\lvert e^{i\frac{t}{\lambda} \Delta}f(x) \right\rvert \, d\nu(x,t) \\
\lesssim  C_{\alpha,s} \lambda^s\|f\|_{2}\left(\int_{B_n(0,1)} \int_{G(\theta)^{\perp}} \left\lvert \xi \right\rvert^{\alpha-n} \left\lvert \widehat{\nu}(\xi) \right\rvert^2 \, d\mathcal{H}^n(\xi) \, d\mathcal{H}^n(\theta)\right)^{1/2}. \end{multline*}
This can be written as 
\begin{multline*} \int \left\lvert e^{i\frac{t}{\lambda} \Delta}f(x) \right\rvert \, d\nu(x,t) \\
\lesssim C_{\alpha,s}  \lambda^s\|f\|_{2} \left(\int_{G(B_n(0,1))} \int_{v^{\perp}} \left\lvert \xi \right\rvert^{\alpha-n} \left\lvert \widehat{\nu}(\xi) \right\rvert^2 \, d\mathcal{H}^n(\xi) \, d(G_{\sharp}\mathcal{H}^n)(v)\right)^{1/2}. \end{multline*}
Since $G$ is Lipschitz, $G_{\sharp}\mathcal{H}^n \lesssim \mathcal{H}^n \restriction_{S^n}$ on $G(B_n(0,1))$, so the above becomes 
\begin{equation} \label{aftercs}  \int \left\lvert e^{i\frac{t}{\lambda} \Delta}f(x) \right\rvert \, d\nu(x,t) \lesssim C_{\alpha,s} \lambda^s \|f\|_{2}\left(\int_{S^n} \int_{v^{\perp}} \left\lvert \xi \right\rvert^{\alpha-n} \left\lvert \widehat{\nu}(\xi) \right\rvert^2 \, d\mathcal{H}^n(\xi)\, d\mathcal{H}^n(v)\right)^{1/2}. \end{equation}
By the X-ray transform identity of Solmon \eqref{solmonidentity} with $k=1$,
\begin{multline*} \int_{S^n} \int_{v^{\perp}} \left\lvert \xi \right\rvert^{\alpha-n} \left\lvert \widehat{\nu}(\xi) \right\rvert^2 \, d\mathcal{H}^n(\xi) \, d\mathcal{H}^n(v) \\=\int_{S^n} \int_{v^{\perp}} |\xi| \left(\left\lvert \xi \right\rvert^{\alpha-(n+1)} \left\lvert \widehat{\nu}(\xi) \right\rvert^2\right) \, d\mathcal{H}^n(\xi) \, d\mathcal{H}^n(v) \\
= c_n \int_{\mathbb{R}^{n+1}} \left\lvert \xi \right\rvert^{\alpha-(n+1)} \left\lvert \widehat{\nu}(\xi) \right\rvert^2 \, d\xi = c_nI_{\alpha}(\nu). \end{multline*}
Combining with \eqref{aftercs} yields 
\[ \int \left\lvert e^{i\frac{t}{\lambda} \Delta}f(x) \right\rvert \, d\nu(x,t) \lesssim C_{\alpha,s} \lambda^s \|f\|_{2}I_{\alpha}(\nu)^{1/2}.  \] 
This verifies \eqref{goal2} and finishes the proof. 
\end{proof} 

\begin{remark} The proof shows that in \eqref{goal}, $I_{\alpha}(\nu)$ could be replaced by 
\[ \inf_{v \in S^{n} : |v-e_{n+1}| < 1/2} I_{\alpha}(P_{v^{\perp} \sharp} \nu), \]
 where $I_{\alpha}(P_{v^{\perp} \sharp} \nu)$ refers to $n$-dimensional Fourier energy, which may be finite when $I_{\alpha}(\nu)$ is not. For example, to disprove the endpoint case $\alpha = n$ and $s = \frac{n}{2(n+1)}$, it would suffice to find a $\nu$ which has one of its projections in $L^2$ such that \eqref{goal} fails.   \end{remark}

\begin{remark} The proof generalises to $L^q$ norms on the left-hand side, with $q>1$, though the optimal $s$ is matches the $q=1$ case for all $1 \leq q \leq 2$ (see \cite[Appendix~C]{BBCR}). The dependence on the measure in the right-hand side has to be updated to $I_{\alpha}(\mu)^{1/2}\|\mu\|^{\frac{1}{q} - 1}$ (which can be thought of as $c_{\alpha}(\mu)^{1/2} \|\mu\|^{ \frac{1}{q} - \frac{1}{2}}$) when $1 \leq q \leq 2$, or $\left(\frac{I_{\alpha}(\mu)}{\|\mu\|}\right)^{1/q}$ (which can be thought of as $c_{\alpha}(\mu)^{1/q}$) when $q \geq 2$.   \end{remark}

\begin{sloppypar}
Here it is shown how the example from \cite{BBCRV, du, lucarogers2} fits with the above when showing that $s \geq n/(2(n+1))$ is necessary.  
\begin{proof}[Re-proof of Bourgain's necessary condition $s \geq \frac{n}{2(n+1)}$] By Stein's maximal principle (see \cite[Appendix~A]{pierce}), it suffices to show that if the maximal operator $f \mapsto  \sup_{0 < t < 1} \left\lvert e^{it \Delta} f(x) \right\rvert$ is bounded from $H^{s}$ to $L^{2, \infty}(B_n(0,1))$, then $s \geq \frac{n}{2(n+1)}$. Assuming this boundedness, 
\[ \sup\left\{   \left\lVert \sup_{0 < t < 1} e^{it \Delta} f \right\rVert_{L^{2, \infty}(\mathbb{R}^n) } : f \in \mathcal{S}(\mathbb{R}^n), \quad \|f\|_{H^s} \leq 1\right\} < \infty. \]
By Cauchy-Schwarz, this implies that
\begin{multline*} \sup\bigg\{ \left\lVert \sup_{0 < t < 1} e^{it \Delta} f \right\rVert_{L^{1, \infty}(\mu)} :\\
 \supp \mu \subseteq B_{n}(0,1), \quad \|\widehat{\mu}\|_2 \leq 1, \quad f \in \mathcal{S}(\mathbb{R}^n), \quad \|f\|_{H^s} \leq 1 \bigg\} < \infty, \end{multline*}
where it was used that the $n$-dimensional Fourier energy of a measure $\mu$ in $\mathbb{R}^n$ equals the $L^2$ norm of $\widehat{\mu}$, which is finite if and only if $\mu \in L^2$ (see \cite[Theorem~3.3]{mattilafourier}). Therefore, the weak version of \eqref{maxassumed2} holds for all $\lambda \geq 1$, and so by Theorem~\ref{equivalence} the weak version of \eqref{goal2} holds for all $\lambda \geq 1$. Therefore, there is a constant $C$ such that for all $\lambda \geq 1$ and $M >0$, 
\begin{equation} \label{weakestimate} M\nu\left\{ (x,t) : \left\lvert e^{i\frac{t}{\lambda} \Delta} f (x)\right\rvert > M \right\} \leq C I_{n}(\nu)^{1/2} \lambda^s \|f\|_2 , \end{equation}
for all $\supp \widehat{f} \subseteq B_n(0, \lambda)$, for all $\nu \in \mathcal{M}\left( B_{n+1}(0,1) \right)$. 

Let $\sigma = \frac{1}{2(n+1)}$.  Let $R$ be large, and let $\varepsilon>0$ be a sufficiently small absolute constant. Let 
\[ \Omega = \mathcal{N}_{\varepsilon} \left(2\pi R^{1-\sigma} \mathbb{Z}^{n-1} \cap B_{n-1}(0,R) \right), \]
and 
\[ \Lambda = \left(\mathcal{N}_{\varepsilon R^{-1} } \left(R^{\sigma-1} \mathbb{Z}^{n-1} \times R^{2\sigma-1} \mathbb{Z}\cap B_{n}(0,2)\right)\right). \]
In \cite{BBCRV}, it was shown that 
\[ \left\lvert e^{i \frac{t}{R} \Delta} \widehat{\chi_{\Omega}} \right\rvert \gtrsim  \mathcal{H}^{n-1}(\Omega), \qquad \forall \, (x',t) \in \Lambda, \]
which follows by checking that the phase of the exponential is within a $\sim \varepsilon$ neighbourhood of an integer multiple of $2\pi i$ whenever $(x',t) \in \Lambda$. Let $\widehat{f_1}(\xi) = R^{-1/2}\varepsilon^{-1}\varphi(R^{-1/2} \varepsilon^{-1} \xi)$ with $\varphi$ a smooth bump function $\sim 1$ on $[-1,1]$ and vanishing outside $\left[-2 ,2 \right]$, with $\int \varphi = 1$. Take $\widehat{f}(\xi_1, \xi') = \widehat{f_1}(\xi_1) \chi_{\Omega}(\xi')$. Then 
\begin{multline*} \left\lvert e^{i\frac{t}{R} \Delta} f(x_1, x', t) \right\rvert = \left\lvert e^{i\frac{t}{R}\Delta_1}f_1(x_1,t) \right\rvert \left\lvert e^{i\frac{t}{R} \Delta_{n-1}}\widecheck{\chi_{\Omega}}(x',t) \right\rvert \\
\gtrsim  \mathcal{H}^{n-1}(\Omega), \quad |x_1| \leq  R^{-1/2}, \qquad (x',t) \in \Lambda. \end{multline*}
Let $\nu$ be the probability measure 
\[ \nu = cR \chi_{[- R^{-1/2},R^{1/2} ] } \times \chi_{\Lambda}, \]
with $c$ a constant $\sim 1$. Taking $M \sim \mathcal{H}^{n-1}(\Omega)$ in \eqref{weakestimate} and $\lambda \sim R$, using that $\supp \widehat{f} \subseteq B_n(0,10R)$, gives
\begin{equation} \label{weakcorollary}  \mathcal{H}^{n-1}(\Omega) \lesssim I_{n}(\nu)^{1/2} R^s \|f\|_2. \end{equation}
It was shown in \cite{du} that $c_n(\nu) \lesssim 1$, as can be checked by partitioning the range of $r$ with the points $\{0, R^{-1}, R^{-1+\sigma}, R^{-1+2\sigma}, R^{-1/2},1\}$ and checking the 5 ranges separately. One finds that $\mu(B(x,r)) \sim r^n$ for all $x \in \supp \mu$ for all $R^{-1/2} \leq r \leq 1$ or $r = R^{-1}$, with $\mu(B(x,r)) \ll r^n$ in the other ranges. It follows that $I_n(\mu) \sim \log R$. The definition of $f$ yields that $\|f\|_2 \sim R^{-1/4} \mathcal{H}_{n-1}(\Omega)^{1/2}$. Combining with \eqref{weakcorollary} gives 
\[ \mathcal{H}^{n-1}(\Omega) \lesssim (\log R)^{1/2} R^s R^{-1/4} \mathcal{H}^{n-1}(\Omega)^{1/2}, \]
or 
\[ \mathcal{H}^{n-1}(\Omega) \lesssim (\log R) R^{2s - \frac{1}{2}}. \]
Substituting $\mathcal{H}^{n-1}(\Omega) \sim R^{\sigma(n-1)}$ and $\sigma = \frac{1}{2(n+1)}$ into the above and letting $R \to \infty$ gives $\frac{n-1}{2(n+1)} \leq 2s-\frac{1}{2}$, which rearranges to $s \geq \frac{n}{2(n+1)}$. 
\end{proof} \end{sloppypar}


\section{Proof for the wave equation} \label{wavesection} Let 
\[ e^{it \sqrt{-\Delta}}g(x) = \frac{1}{(2\pi)^{n}} \int e^{i \left( \langle x, \xi \rangle + t|\xi| \right) } \widehat{g}(\xi). \]
In this section, for $\alpha >0$ let $s_n^{\max}(\alpha)$ be the infimum over all $s \geq 0$ with the property that for all $R \geq 1$,
\begin{equation} \label{smaxdefnwave} \int \sup_{0 < t < 1} \left\lvert e^{it \sqrt{-\Delta}}g(x) \right\rvert \, d\mu(x) 
\leq C_{\alpha,s} I_{\alpha}(\mu)^{1/2} R^{s}\|g\|_2,\end{equation}
for all Schwartz $g$ with $\supp \widehat{g} \subseteq B_n(0,R) \setminus B_n(0, R/2)$, for all $\mu \in \mathcal{M}(B_n(0,1))$. Similarly, let $s_n(\alpha)$ be the infimum over all $s \geq 0$ with the property that for all $R \geq 1$,
\begin{equation} \label{smaxdefnwave2} \int \sup_{0 < t < 1} \left\lvert e^{it \sqrt{-\Delta}}g(x) \right\rvert \, d\nu(x) 
\leq C_{\alpha,s} I_{\alpha}(\nu)^{1/2} R^{s}\|g\|_2,\end{equation}
for all Schwartz $g$ with $\supp \widehat{g} \subseteq B_n(0,R) \setminus B_n(0, R/2)$, for all $\nu \in \mathcal{M}(B_{n+1}(0,1))$. By duality, $\beta(\alpha, \Gamma^n) = n - 2s_n(\alpha)$ for $0 < \alpha \leq n+1$, where $\beta(\alpha, \Gamma^n)$ is the supremum over all $\beta>0$ with the property that for all $R \geq 1$ and $\nu \in \mathcal{M}(B_{n+1}(0,1))$, 
\[ \int_{B_n(0,1) \setminus B_n(0,1/2)} \left\lvert \widehat{\nu}(R(\xi, |\xi| )) \right\rvert^2 \, d\xi \leq C_{\beta} I_{\alpha}(\nu) R^{-\beta}.  \]
The quantity $s_n^{\max}(\alpha)$ was considered in \cite{hamkolee} where it was shown that $s_n^{\max}(\alpha) \leq s_n(\alpha)$ for all $0 < \alpha \leq n$. The following is the wave equation analogue of Theorem~\ref{equivalence}. 
\begin{theorem} \label{equivalence2} Let $n \geq 1$.  Let $\alpha>0$ and let $s \geq  0$. There is a constant $K$  depending only on $s$ and $n$ such that following holds. 
\begin{enumerate}[(i)] 
\item \label{sufficientcondition2}   If 
\begin{multline} \label{maxassumedwave} \int \sup_{0 < t < 1} \left\lvert e^{i t \sqrt{-\Delta}}g \right\rvert \, d\mu(x) \leq C_{\alpha,s}I_{\alpha}(\mu)^{1/2} \left\lVert g \right\rVert_{H^s(\mathbb{R}^n)} \\
 \forall \mu \in \mathcal{M}\left(B_n(0,1)\right) \quad \forall g \in \mathcal{S}(\mathbb{R}^n), \end{multline}
then
\begin{multline} \label{goalwave} \int \left\lvert e^{i t \sqrt{-\Delta}}f \right\rvert \, d\nu(x,t) \leq K C_{\alpha,s}I_{\alpha}(\nu)^{1/2} \left\lVert f \right\rVert_{H^s(\mathbb{R}^n)}\\
 \forall \nu \in \mathcal{M}\left(B_{n+1}(0,1)\right) \quad \forall f \in \mathcal{S}(\mathbb{R})^n. \end{multline}

\item \label{localversionwave} (Local version)  If $\lambda\geq 1$ and
\begin{multline} \label{maxassumedwave2} \int \sup_{0 < t < 1} \left\lvert e^{i t \sqrt{-\Delta}}g \right\rvert \, d\mu(x) \leq C_{\alpha,s}I_{\alpha}(\mu)^{1/2} \lambda^s \left\lVert g \right\rVert_{2}\\
\forall \mu \in \mathcal{M}\left(B_n(0,1)\right) \quad \forall g \in \mathcal{S}(\mathbb{R}^n) \text{ with } \supp \widehat{g} \subseteq B_n(0, \lambda), \end{multline}
 then 
\begin{multline} \label{goalwave2} \int \left\lvert e^{i t \sqrt{-\Delta}}f \right\rvert \, d\nu(x,t) \leq K C_{\alpha,s} I_{\alpha}(\nu)^{1/2} \lambda^s \left\lVert f \right\rVert_2\\
\\
 \forall \nu \in \mathcal{M}\left(B_{n+1}(0,1)\right) \quad \forall f \in \mathcal{S}(\mathbb{R})^n \text{ with }  \supp \widehat{f} \subseteq B_n(0, \lambda).\end{multline} 
\item \label{nontrivialwave1} For all $\alpha \in (0,n]$, $s_n(\alpha) = s_{n}^{\max}(\alpha)$. If $n <\alpha \leq n+1$ then $s_n^{\max}(\alpha) \geq s_n(\alpha)$.  
\end{enumerate}
 \end{theorem}

\begin{proof} The proof reduces to the idea in Theorem~\eqref{equivalence} by using that paraboloids are conic sections, so only the differences for the wave equation will be given. It will be shown that \eqref{maxassumedwave2} implies \eqref{goalwave2}, so assume \eqref{maxassumedwave2} holds for some $s$. 

Let $f$ be a function on the cone, and denote
\[ Ef(x,t) = \int_{\mathbb{R}^n} e^{ i \langle (\xi,|\xi|), (x,t) \rangle } f  (\xi, |\xi|) \, d\xi, \]
Suppose that $f$ is supported on $\xi_1 \geq 0$, where $\xi = (\xi_1, \dotsc, \xi_{n+1})$. Write $\xi_{n+1} = |\xi|$, and define $U$ and $\eta$ by $U\xi = \eta$, where
\[ \eta_1 = (\xi_1 + \xi_{n+1})/ \sqrt{2},  \qquad \eta_{n+1} = (-\xi_1 + \xi_{n+1})/ \sqrt{2},\]
and $(\eta_2, \dotsc, \eta_n) = (\xi_2, \dotsc, \xi_n)$. Then the cone $\xi_{n+1} = |(\xi_1, \dotsc, \xi_n)|$ becomes $\eta_{n+1} = \frac{|(\eta_2, \dotsc, \eta_{n})|^2}{2\eta_1}$ in the new coordinates, and the region $\xi_1 \geq 0$ becomes $\eta_1 \geq |\eta'|/ \sqrt{2}$. By a change of variables
\begin{equation} \label{cov1} Ef(x,t) = \int e^{ i \left\langle \left(\eta_1, \eta', \frac{|\eta'|^2}{2\eta_1}\right), \left(\frac{x_1+t}{\sqrt{2}},x', \frac{-x_1+t}{\sqrt{2}}\right) \right\rangle } g(\eta) \, d\eta_1 \, \dotsb \, d\eta_n, \end{equation}
where $g(\eta_1,\eta') = f(U^*(\eta_1, \eta', \eta_{n+1}))) J(\eta)$, where $1/ \sqrt{2} \leq J(\eta) \leq \sqrt{2}$, $U^*$ is the adjoint or transpose of $U$, and $\eta_{n+1} = \frac{|(\eta_2, \dotsc, \eta_{n})|^2}{2\eta_1}$. The function $J$ is the Jacobian defined by the change of variables
\[ \int_{\mathbb{R}^n} G(\eta) J(\eta) \, d\eta = \int_{\mathbb{R}^n} G\left( \frac{\xi_1+|\xi|}{\sqrt{2}}, \xi'\right) \, d\xi, \]
and one can compute 
\[ J(\eta) =  \frac{1}{\sqrt{2}} \left( 1+ \frac{|\eta'|^2}{ 2 \eta_1^2} \right) = \frac{ |\xi|\sqrt{2}}{ \xi_1 + |\xi| }. \]
Define, for the function $g$ on $\mathbb{R}^n$ above, for $(\eta_1, \eta') \in \mathbb{R} \times \mathbb{R}^{n-1}$,
\[ g(\eta) = \lambda g_{\lambda,\theta}(\lambda \eta_1, \eta'  - \theta \eta_1 ), \]
where 
\begin{equation} \label{domain} (\lambda, \theta) \in \left\{ (\lambda, \theta) \in \mathbb{R} \times \mathbb{R}^{n-1} : \lambda \geq 1 + \frac{|\theta|}{\sqrt{2}} \right\}, \end{equation}
and correspondingly 
\[ g_{\lambda, \theta}(\eta_1,\eta') =: f_{\lambda, \theta}(U^*(\eta_1, \eta', \eta_{n+1})) J(\eta). \]
Note that $\lambda \eta_1  \geq |\eta'-\theta\eta_1|/ \sqrt{2}$ by the triangle inequality and since $(\lambda, \theta)$ lie in the domain in \eqref{domain}.  Therefore, the Jacobian factor coming from $ g_{\lambda, \theta}$ in the above formula for $f_{\lambda, \theta}$ also has values in $\left[1/\sqrt{2}, \sqrt{2}\right]$, and thus $\|f_{\lambda, \theta}\|_{L^2(\Gamma^n)} \sim \|f\|_{L^2(\Gamma^n)}$.
By a change of variables applied to \eqref{cov1}, 
\[ Ef(x,t) =  \int e^{ i \left\langle \left(\eta_1/\lambda,\eta'+(\eta_1/\lambda)\theta, \frac{ |\eta'+\theta (\eta_1/\lambda) |^2}{2\eta_1/\lambda }\right), \left(\frac{x_1+t}{\sqrt{2}}, x', \frac{-x_1+t}{\sqrt{2}}\right) \right\rangle } g_{\lambda, \theta}(\eta) \, d\eta_1 \, d\eta'. \]
This becomes
\begin{multline*} Ef(x,t) = \\ 
\int e^{ i \left\langle \left(\eta_1,\eta', \frac{|\eta' |^2}{2\eta_1}\right), \left(\frac{x_1+t}{\lambda \sqrt{2}} + \frac{|\theta|^2(-x_1+t)}{2\lambda \sqrt{2}} + \frac{\langle \theta, x' \rangle}{\lambda}, x'+\frac{t-x_1}{\sqrt{2}}\theta, \frac{\lambda(-x_1+t)}{\sqrt{2}}\right) \right\rangle } g_{\lambda, \theta}(\eta) \, d\eta_1  \, d\eta' . \end{multline*}
Converting back to the old variable $\xi$ by undoing the change of variables above results in
\begin{multline*} Ef(x,t) = \\ Ef_{\lambda, \theta}\left(U^*\left(\frac{x_1+t}{\lambda \sqrt{2}} + \frac{|\theta|^2(-x_1+t)}{2\lambda \sqrt{2}} + \frac{\langle \theta, x' \rangle}{\lambda},  x'+\frac{t-x_1}{\sqrt{2}}\theta, \frac{\lambda(-x_1+t)}{\sqrt{2}}\right) \right), \end{multline*}
where $U^*(x,x',t) = \left( \frac{x_1-t}{\sqrt{2}}, x', \frac{x_1+t}{\sqrt{2}} \right)$. This can be written as 
\begin{equation} \label{Efthetaformula} Ef(x,t) = Ef_{\lambda, \theta}\left( T_{\lambda, \theta}(x,t) \right),  \end{equation}
where
\begin{multline*} T_{\lambda, \theta}(x,t) = \bigg( \frac{x_1}{2}\left( \lambda+ \frac{1}{\lambda} \right)  - \frac{t}{2} \left( \lambda - \frac{1}{\lambda} \right) + \frac{|\theta|^2 (t-x_1)}{4\lambda} + \frac{\langle \theta, x' \rangle}{\lambda\sqrt{2}}, \\
x' + \frac{t-x_1}{\sqrt{2}} \theta, \quad  \frac{t}{2}\left({\lambda}+\frac{1}{\lambda} \right) -\frac{x}{2} \left( \lambda - \frac{1}{\lambda} \right)+ \frac{|\theta|^2 (t-x_1)}{4\lambda}  + \frac{\langle \theta, x' \rangle}{\lambda\sqrt{2}} \bigg).   \end{multline*}
Let $\nu_{\lambda, \theta} = T_{\lambda, \theta, \sharp} \nu$, where $\nu$ is a finite Borel measure on $B_{n+1}(0,1)$. Let $R \geq 1$. Without loss of generality it may be assumed that $\nu$ is supported in $B_{n+1}(0,1/100)$. Let $f$ be a function on the cone supported in $B(0,R) \setminus B(0,R/2)$. By the above, 
\begin{align*}  \int \left\lvert Ef(x,t) \right\rvert \, d\nu(x,t) &= \int \left\lvert Ef_{\lambda, \theta}(T_{\lambda, \theta}(x,t)) \right\rvert \, d\nu\\
  &= \int \left\lvert Ef(x,t) \right\rvert \, dT_{\lambda, \theta \sharp}\nu(x,t) \\
	&\leq \int \sup_{t \in [-1,1]} \left\lvert Ef_{\lambda, \theta}(x,t) \right\rvert \, d\left(\pi_{\sharp} \nu_{\lambda, \theta} \right).\end{align*}
%
Let $\Pi_{\lambda, \theta} = \pi \circ T_{\lambda, \theta}$, so that $\pi_{\sharp} \nu_{\lambda, \theta}  = \Pi_{\lambda, \theta \sharp} \nu$. Rewrite the above as
\begin{equation} \label{projectsmax} \int \left\lvert Ef(x,t) \right\rvert \, d\nu(x,t) \leq  \int \sup_{t \in [-1,1]} \left\lvert Ef_{\lambda, \theta}(x,t) \right\rvert \, d\left(\Pi_{\lambda, \theta \sharp} \nu \right). \end{equation}
If $h$ is a Schwartz function with $\supp \widehat{h} \subseteq B_n(0,\lambda) \setminus B_n(0, \lambda/2)$, let $f$ be the lift of $\widehat{h}$ into the cone. Let $h_{\lambda, \theta}$ be such that $f_{\lambda, \theta}$ is the lift of $\widehat{h_{\lambda, \theta}}$. Then $\widehat{h_{\lambda, \theta}}$ has support in $B_n(0, C\lambda) \setminus B_n(0, \lambda/C)$ for a sufficiently large constant $C$, for $\lambda \geq 100$. By \eqref{projectsmax} and the assumed \eqref{maxassumedwave2}, 
\begin{equation} \label{pause2}\int \left\lvert e^{it \sqrt{-\Delta} } h(x) \right\rvert \, d\nu(x,t)  \leq C_{\alpha,s} \sqrt{I_{\alpha}\left( \Pi_{\lambda, \theta \sharp} \nu\right)}  \lambda^s \|h\|_2.\end{equation}
The function $\Pi_{\lambda, \theta}$ is a linear map which can be identified with the $n \times (n+1)$ matrix
 \begin{equation} \label{projformula} \Pi_{\lambda, \theta} = \pi \circ T_{\lambda, \theta} \\=\begin{pmatrix} \frac{1}{2}\left(\lambda + \frac{1}{\lambda}\right) - \frac{|\theta|^2}{4\lambda} & \frac{\theta}{\lambda \sqrt{2}} & - \frac{1}{2}\left( \lambda- \frac{1}{\lambda} \right) +  \frac{|\theta|^2}{4\lambda} \\
 - \frac{\theta}{\sqrt{2}} & I_{n-1} & \frac{\theta}{\sqrt{2}} \end{pmatrix}, \end{equation}
which has full rank $n$ for $\lambda \neq 0$. The kernel of $\Pi_{\lambda, \theta}$
 is the span of the unit vector
 \begin{equation} \label{kernelformula} G(\theta, \lambda) := \frac{1}{\sqrt{ \frac{1}{2} + \frac{ \left\lvert \theta\right\rvert^2}{2}  + 2\left( \frac{ \left\lvert \theta \right\rvert^2}{4} + \frac{\lambda^2}{2} \right)^2 }}\begin{pmatrix} \frac{1}{2} - \frac{\left\lvert \theta \right\rvert^2}{4} - \frac{\lambda^2}{2} \\ \frac{\theta}{\sqrt{2}}  \\ -\frac{1}{2} -\frac{ \left\lvert \theta\right\rvert^2 }{4} - \frac{\lambda^2}{2} \end{pmatrix}. \end{equation}
 The map $G : \mathbb{R}^n \setminus \{\lambda = 0\} \to S^{n}$ is a local diffeomorphism everywhere, and is bijective when restricted to $\lambda >0$. The bijective condition is straightforward to check. 
To see the diffeomorphism condition, the derivative of the function $F: \mathbb{R}^{n} \to \mathbb{R}^{n+1}$ given by 
 \[ F(\theta, \lambda) = \begin{pmatrix} \frac{1}{2} - \frac{\left\lvert \theta \right\rvert^2}{4} - \frac{\lambda^2}{2} \\ \frac{\theta}{\sqrt{2}}  \\ -\frac{1}{2} -\frac{ \left\lvert \theta\right\rvert^2 }{4} - \frac{\lambda^2}{2} \end{pmatrix}. \]
 is 
\[ DF(\lambda, \theta) = \begin{pmatrix} - \frac{\theta}{2} & - \lambda \\ \frac{1}{\sqrt{2}} I_{n-1} & 0 \\ - \frac{\theta}{2} & - \lambda \end{pmatrix}. \]
The map $G$ is a local diffeomorphism at $(\theta, \lambda)$ provided $DF(\theta, \lambda)$ has full rank $n$ and $F(\theta, \lambda)$ is not a linear combination of the columns of $DF(\theta, \lambda)$. The full rank condition holds since $\lambda \neq 0$, and $F$ cannot be a linear combination of the columns of $DF$ since the first and last rows of $DF$ are identical, but the first and last entries of $F$ are distinct. 
Thus $G$ is a local diffeomorphism.

By the QR decomposition of $\Pi_{\lambda, \theta}^*$, for $\lambda \neq 0$, 
\[ \Pi_{\lambda, \theta} = \pi \circ T_{\lambda, \theta} = L_{\lambda, \theta} Q_{\lambda, \theta}, \]
 where $L_{\lambda, \theta}$ is an $n \times n$ lower triangular matrix depending continuously on $(\lambda, \theta) \in \mathbb{R}^n \setminus \{\lambda  = 0\}$, with positive entries on the diagonal, and $Q_{\lambda, \theta}$ is an $n \times (n+1)$-matrix such that the rows of $Q_{\lambda, \theta}$ form an orthonormal basis for the orthogonal complement of $G(\theta, \lambda)$ in $\mathbb{R}^{n+1}$.

 
In particular, since $L_{\lambda, \theta}$ is continuous, for any bounded set
\begin{equation} \label{compactK} K \subseteq \{ (\lambda, \theta) \in \mathbb{R} \times \mathbb{R}^{n-1} : |\lambda| > 1 + |\theta|/\sqrt{2}\}, \end{equation}
one has \begin{equation} \label{localdiffeo1} I_{\alpha}\left(\Pi_{\lambda, \theta \sharp} \nu \right) \sim_K I_{\alpha}\left(Q_{\lambda, \theta \sharp} \nu \right)  = \int_{G(\theta, \lambda)^{\perp}} \left\lvert \xi \right\rvert^{\alpha-n} \left\lvert \widehat{\nu}(\xi) \right\rvert^2 \, d\mathcal{H}^n(\xi). \end{equation}
 where $P_{v^{\perp}}$ denotes orthogonal projection to the orthogonal complement of $v$. Take $K= (2,3) \times B_{n-1}(0,1)$. Since $G: \mathbb{R}^n \to S^n$ is a diffeomorphism when restricted to a compact set containing $K$,
\begin{equation} \label{localdiffeo2}  \int_K I_{\alpha}\left(\Pi_{\lambda, \theta \sharp} \nu \right) \, d\mathcal{H}^n(\lambda, \xi)  \lesssim \int_{S^n} \int_{v^{\perp}} \left\lvert \xi \right\rvert^{\alpha-n} \left\lvert \widehat{\nu}(\xi) \right\rvert^2 \, d\mathcal{H}^n(\xi) \, d\mathcal{H}^n(v).  \end{equation}
After integrating \eqref{pause2} over $K$ and applying Cauchy-Schwarz followed by \eqref{localdiffeo2}, the rest of the proof is almost identical similar to the Schrödinger case (from around \eqref{aftersimplify}), so it is omitted. 
\end{proof}

\begin{remark} The method for the Schrödinger equation would extend easily to the hyperbolic Schrödinger equation.  It would likely also extend to the Helmholtz equation and spherical averages, using rotations instead of translations, but I do not know if the pointwise convergence problem for the Helmholtz equation has been considered in the literature.  \end{remark}

\bibliographystyle{alpha}
\bibliography{main}

\end{document}